\documentclass[abstracton,a4paper,headings=normal]{scrartcl}

\usepackage{amsfonts,amsthm,amssymb,amsmath}

\usepackage{enumitem}
\usepackage{hyperref}

\usepackage{tikz}
\usetikzlibrary{shapes}
\usetikzlibrary{arrows.meta}

\newtheorem{thm}{Theorem}
\newtheorem{lem}[thm]{Lemma}
\newtheorem{prp}[thm]{Proposition}
\newtheorem{con}[thm]{Conjecture}
\theoremstyle{definition}
\newtheorem{dfn}[thm]{Definition}
\newtheorem{obs}[thm]{Observation}

\newcommand{\Cay}{\operatorname{Cay}}
\newcommand{\id}{\operatorname{id}}

\newcommand{\gen}[1]{
\ensuremath{\mathord{\raisebox{-.3ex}{\begin{tikzpicture}[line width = .12ex,x=.8ex,y=.8ex]
\phantom{\path[draw] (0,0) circle (1.1);}%spacing for superscripts etc.
\path[draw,line width = .10ex] (0,0) circle (1);
\path[draw,-{Triangle[scale=.7]}] (#1:-.8)--(#1:.8);
\end{tikzpicture}}}}}

\newcommand{\rgen}{\gen{0}}
\newcommand{\ugen}{\gen{90}}
\newcommand{\lgen}{\gen{180}}
\newcommand{\dgen}{\gen{270}}
\newcommand{\agen}{\gen{135}}

\usepackage{authblk}

\begin{document}

\title{Hamiltonian decompositions of 4-regular Cayley graphs of infinite abelian groups}
\author{Joshua Erde}
\author{Florian Lehner\thanks{Florian Lehner acknowledges the support of the Austrian Science Fund (FWF) through grants J~3850-N32 and P~31889-N35.}}
\affil{\small Graz University of Technology, \\
Institute of Discrete Mathematics, \\
Steyrergasse 30, \\
8010 Graz, Austria.}

\maketitle
\begin{abstract}
A well-known conjecture of Alspach says that every $2k$-regular Cayley graph of an abelian group can be decomposed into Hamiltonian cycles. We consider an analogous question for infinite abelian groups. In this setting one natural analogue of a Hamiltonian cycle is a spanning double-ray. However, a naive generalisation of Alspach's conjecture fails to hold in this setting due to the existence of $2k$-regular Cayley graphs with finite cuts $F$ where $|F|$ and $k$ differ in parity, which necessarily preclude the existence of a decomposition into spanning double-rays.

We show that every $4$-regular Cayley graph of an infinite abelian group all of whose finite cuts are even can be decomposed into spanning double-rays, and so characterise when such decompositions exist. We also characterise when such graphs can be decomposed either into Hamiltonian circles, a more topological generalisation of a Hamiltonian cycle in infinite graphs, or into a Hamiltonian circle and a spanning double-ray.
\end{abstract}

\section{Introduction}
A \emph{Hamiltonian cycle} in a finite graph $G$ is a cycle which includes every vertex of the graph; a \emph{Hamiltonian decomposition} is a partition of the edge set of $G$ into disjoint sets $E=E_1 \uplus E_2 \uplus \cdots \uplus E_r$, 
where each $E_i$ is a Hamiltonian cycle in $G$. One of the earliest results in graph theory is a theorem of Walecki from 1890 stating that every finite complete graph of odd order has a Hamiltonian decomposition (see, for example \cite{A08}). Since then, Hamiltonian decompositions of various classes of graphs have been studied, the survey of Alspach, Bermond and Sotteau \cite{ABS90} gives an overview. 

It is a folklore result that every Cayley graph of an finite abelian group contains a Hamiltonian cycle, hence it is natural to ask for which of them we can find a Hamiltonian decomposition. Sometimes this fails for parity reasons. Indeed, if a graph has a Hamiltonian decomposition then it must be $2k$-regular for some $k$. It is a long-standing conjecture of Alspach that for Cayley graphs of finite abelian groups this is the only thing that can go wrong.

\begin{con}[Alspach \cite{alspach1984research,A85}]\label{c:ab}
If $\Gamma$ is a finite abelian group and $S$ generates $\Gamma$, then the Cayley graph
$\Cay(\Gamma,S)$ has a Hamiltonian decomposition, provided that it is $2k$-regular for some $k$.
\end{con}

Not much is known about this conjecture. If $k=1$, then it trivially holds, and in case $k=2$ it was proved by Bermond, Favaron and Meheo \cite{BFM89}. However, even the case $k=3$ is still open, although partial results towards this case were given by Liu \cite{L94} and Westlund \cite{W12}. Liu also showed \cite{L96,L03} that the conjecture holds for any $k$ when $S$ is a minimal generating set.

While the previous results all concerned finite graphs, Hamiltonian cycles have also been considered in infinite graphs. It is not immediate what the correct generalisation of a Hamiltonian cycle to an infinite graph should be. One natural structure to consider is a spanning \emph{double-ray}, an infinite connected graph in which each vertex has degree two, which we will refer to as a \emph{Hamiltonian double-ray}.

Nash-Williams \cite{N59} showed that every connected Cayley graph of a finitely generated infinite abelian group contains a Hamiltonian double-ray, and together with a result of Witte \cite{W90} this then implies that every connected Cayley graph of a finitely generated infinite abelian group with infinite degree (that is, $S$ is infinite) has a decomposition into Hamiltonian double-rays. More recently, the authors and Pitz \cite{ELP20} showed that if $\Gamma$ is a finitely generated abelian group, every element of the finite generating set $S$ has infinite order, and $\Cay(\Gamma,S)$ is one-ended, then it has a decomposition into Hamiltonian double-rays. 

Besides $G$ having to be $2k$-regular for some $k$, there is another parity obstruction to the existence of a decomposition into Hamiltonian double-rays. 
A \emph{cut} is a partition of the vertex set into two parts called the \emph{sides} of the cut; it is called finite if there are finitely many edges connecting the two sides, called \emph{cross edges} or simply \emph{edges} of the cut. If $F$ is a finite cut both of whose sides are infinite, then any Hamiltonian double-ray must contain an odd number of edges of $F$, otherwise it only contains finitely many vertices on one of the sides. So a decomposition into Hamiltonian double-rays can only exist if the number of cross edges of $F$ has the same parity as $k$. 
Note that, for both of the results mentioned above this parity condition  does not play a role since there are no finite cuts with two infinite sides.

In this paper, we restrict our attention to $4$-regular Cayley graphs of infinite abelian groups. For such graphs, we can assume that either $\Gamma = \mathbb{Z}^2$, or $\Gamma = \mathbb{Z}$, or $\Gamma = \mathbb{Z} \oplus \mathbb{Z}_i$ for some $i$, see Proposition \ref{p:classification} and the discussion thereafter. Moreover, for such Cayley graphs the parity condition on finite cuts mentioned above boils down to the following:
\begin{enumerate}[label = (P)]
\item \label{itm:parity} Every finite cut contains an even number of edges.
\end{enumerate}

The existence of a decomposition into Hamiltonian double-rays in the case $\Gamma = \mathbb{Z}^2$ follows from the work of the authors and Pitz in \cite{ELP20}, also see \cite[Proposition 5]{GL18} for a short, direct proof. Bryant, Herke, Maenhaut, and Webb \cite{BHMW18} considered the case $\Gamma = \mathbb{Z}$ and showed among other things that if $S$ is any generating set with $|S|=2$, then $\Cay(\mathbb Z,S)$ has a decomposition into Hamiltonian double-rays if it satisfies \ref{itm:parity}. In this paper we extend this result to groups of the form $\Gamma = \mathbb{Z} \oplus \mathbb{Z}_i$ for some $i$, thus proving the following result.

\begin{thm}
\label{thm:main}
Let $G$ be a connected, $4$-regular Cayley graph of an infinite abelian group which satisfies \ref{itm:parity}, then $G$ has a decomposition into Hamiltonian double-rays.
\end{thm}

Our proof also gives Hamiltonian decompositions for a different notion of infinite Hamiltonian cycles called \emph{Hamiltonian circles}. The notion is based on a topological approach to infinite graph theory, a comprehensive introduction to which can be found in \cite{D11,D10,DS12}. We defer the precise definitions to Section \ref{s:topologicalbackground}, but mention that any Hamiltonian circle meets any finite cut in an even number of edges (see Lemma \ref{l:meeteven}). Thus \ref{itm:parity} is also necessary for a decomposition of a $4$-regular graph into Hamiltonian circles to exist. Once again, it turns out that for $4$-regular Cayley graphs of abelian groups, \ref{itm:parity} is also sufficient.

\begin{thm}
\label{thm:main-circle}
Let $G$ be a connected, $4$-regular Cayley graph of an infinite abelian group which satisfies \ref{itm:parity}, then $G$ has a decomposition into Hamiltonian circles.
\end{thm}

Finally, in case \ref{itm:parity} does not hold, we are able to find a `mixed' decomposition into a Hamiltonian double-ray and a Hamiltonian circle.

\begin{thm}
\label{thm:split}
Let $G$ be a connected, $4$-regular Cayley graph of an infinite abelian group which does not satisfy \ref{itm:parity}, then $G$ has a decomposition into a Hamiltonian double-ray and a Hamiltonian circle.
\end{thm}

\section{Preliminaries}\label{s:prelim}
\subsection{Topological infinite graph theory}\label{s:topologicalbackground}
A graph $G$ is \emph{locally finite} if every vertex has finite degree. A \emph{ray} in a graph is a one-way infinite path, and an \emph{end} of a locally finite graph is an equivalence class of rays under the relation $R_1 \sim R_2$ if for every finite cut $F$, all but finitely many vertices of $R_1$ and $R_2$ lie on the same side of $F$. If we denote by $\Omega$ the set of ends of a graph $G$ then there is a natural topology on the $1$-complex of $G$ together with $\Omega$ which forms a compact topological space known as the \emph{Freudenthal compactification of $G$} which is normally denoted by $|G|$. A \emph{circle} in $G$ is a subspace of $|G|$ homeomorphic to the circle $S_1$. It can be shown that a circle is uniquely defined by the set of edges contained in it, so by a slight abuse of notation we will also call this set of edges a circle. 

It is worth noting that there is an equivalent, combinatorial definition of a circle, generalising the fact that a cycle is an inclusion minimal element of the cycle space of a graph.

\begin{lem}[\cite{DK04}]\label{l:meeteven}
Let $G$ be a locally finite graph. Then a set of edges $C$ is a circle if and only if $C$ meets every finite cut $F$ of $G$ in an even number of edges, and there is no non-empty $C' \subsetneq C$ with this property.
\end{lem}

A \emph{Hamiltonian circle} is a circle which meets every vertex of $G$. It is relatively easy to show that every Hamiltonian circle in a one-ended graph is a spanning double-ray. For two-ended graphs it can be shown that every Hamiltonian circle is a disjoint union of two double-rays which together span $G$, each of which contains a ray to both ends of the graph. However, for our purposes we will only need the converse of both of these statements, that such a subgraph is a Hamiltonian circle, which is a simple consequence of Lemma \ref{l:meeteven} and whose proof we provide for completeness.

\begin{lem}\label{l:circle}
\begin{enumerate} 
\item If $G$ is a locally finite, one-ended graph and $C$ is a spanning double-ray, then $C$ is a Hamiltonian circle.
\item If $G$ is a locally finite, two-ended graph and $C$ is a disjoint union of two double-rays which together span $G$, each of which contains a ray to both ends of the graph, then $C$ is a Hamiltonian circle.
\end{enumerate}
\end{lem}
\begin{proof}
In the case that $G$ is one-ended, every finite cut $F$ has a unique infinite component, which must contain both tails of $C$, and hence $C$ must meet $F$ in an even number edges. For every non-empty strict subset $C'$ of $C$, there is at least one vertex only incident to one edge in $C'$. Thus $C'$ meets the cut with this vertex on one side and all other vertices on the other side in only one edge. It follows from Lemma \ref{l:meeteven} that $C$ is a Hamiltonian circle.

In the case that $G$ is two-ended, every finite cut $F$ has either one, or two infinite components. Let $C_1$ and $C_2$ be the two double-rays forming $C$. If $F$ has one infinite component, then both tails of $C_1$ and $C_2$ are contained in this component, and so both must meet $F$ in an even number of edges. In the second case the two tails of $C_1$ and $C_2$ are contained in different components, and so both must meet $F$ in an odd number of edges. In either case, $C$ meets $F$ in an even number of edges.
If $C' \subsetneq C$ is non-empty, then either there is a vertex incident to only one edge in $C'$, or $C'$ is one of $C_1$ and $C_2$. In the first case we can use the same argument as above to show that $C'$ meets some finite cut in an odd number of edges. Otherwise, let $F$ be a finite cut witnessing the fact that two sub-rays of $C'$ lie in different ends. Then $C'$ must contain infinitely many vertices on both sides of $F$ and thus it contains an odd number of cross edges of $F$. Consequently, by Lemma \ref{l:meeteven}, $C$ is a Hamiltonian circle.
\end{proof}

\subsection{Structure of 4-regular Cayley graphs of abelian groups}
It will be useful to give a classification of the possible graphs that can arise as $4$-regular Cayley graphs of infinite abelian groups.

\begin{dfn}
For any $k\in \mathbb N$ and $l \in \mathbb Z$, the graph $G_{k,l}$ is the graph with
\[
V(G_{k,l}) = \{(m,n)\mid m,n \in \mathbb Z, 0 \leq m < k\},
\]
and whose edge set consists of the following three kinds of edges:
\begin{enumerate}[leftmargin=1.5cm,label={(\alph*)}]
\item\label{e:a} $(m,n)$ to $(m,n+1)$ for $m,n \in \mathbb Z, 0 \leq m < k$,
\item\label{e:b} $(m,n)$ to $(m+1,n)$ for $m,n \in \mathbb Z, 0 \leq m < k-1$, and
\item\label{e:c} $(k-1,n)$ to $(0,n-l)$ for $n \in \mathbb Z$.
\end{enumerate}
\end{dfn}

In all figures throughout this paper, we represent $G_{k,l}$ as follows. We draw every vertex $(m,n)$ at coordinates $(m,n)$ in the plane. Edges of type \ref{e:a} and \ref{e:b} are drawn as straight line segments, edges of type \ref{e:c} are represented by two half edges to the right of $(k-1,n)$ and to the left of $(0,n-l)$. Numbers next to these half edges indicate which of them correspond to the same edge, see for instance Figure \ref{f:Walk-3-1}. We will refer to edges of type \ref{e:a} as \emph{vertical} edges and to edges of types \ref{e:b} and \ref{e:c} as \emph{horizontal} edges.

The aim of this section is to show that any $4$-regular Cayley graph of an abelian group apart from the square grid is in fact of the form $G_{k,l}$ for some $k$ and $l$. To this end, the following observation will be useful.

\begin{obs}
\label{obs:group-graph-iso}
Let $\Gamma, \Delta$ be isomorphic groups. Then every group isomorphism  $\phi\colon \Gamma \to \Delta$ is also a graph isomorphism $\Cay(\Gamma, \{s_1,\dots,s_k\}) \to \Cay(\Delta, \{\phi(s_1),\dots,\phi(s_k)\})$. This also holds for endomorphisms.
\end{obs}

\begin{prp}\label{p:classification}
\label{prp:structure}
If $G$ is a $4$-regular Cayley graph of an infinite abelian group then either $G$ is the square grid, or there exists $k \in \mathbb{N}$ and $l \in \mathbb{Z}$ such that $G \simeq G_{k,l}$.
\end{prp}

\begin{proof}
If $G = \Cay(\Gamma,S)$ is a 4-regular Cayley graph of an abelian group, then $2 \leq |S| \leq 4$. If $|S| = 4$, then all of the generators must be involutions and the group is finite. If $|S|=3$, then two of the generators are involutions and hence the group is either finite (in case the third generator has finite order) or $\Gamma = \mathbb Z \oplus \mathbb Z_2 \oplus \mathbb Z_2 = \mathbb Z \oplus V_4$. In the latter case it is easy to verify that $G = G_{4,0}$ is the only possibility if two involutions of $V_4$ appear in the generating set.

So assume that $S = \{a,b\}$. In this case there is a unique endomorphism $\phi \colon \mathbb Z^2 \to \Gamma$ which maps $(1,0)$ to $a$ and $(0,1)$ to $b$. By the isomorphism theorem, $\Gamma \simeq \mathbb Z^2 / \ker \phi$ and by Observation \ref{obs:group-graph-iso} the respective Cayley graphs are isomorphic as well. So it suffices to study Cayley graphs of groups of the form $\Gamma = \mathbb Z^2 / N$ with generators $(1,0)$ and $(0,1)$, where $N$ is any subgroup of $\mathbb Z^2$.

If $N \simeq \mathbb Z^2$ then $\mathbb Z^2 / N$ is finite. If $N = \{(0,0)\}$, then $\Gamma = \mathbb Z^2$ and the Cayley graph is the square grid. The only remaining case is when $N$ is infinite cyclic, i.e.\ there are $k,l \in \mathbb Z$ such that $N = \{n \cdot (k,l) \mid n \in \mathbb Z\}$. We can without loss of generality assume that $k > 0$---clearly $k$ and $l$ cannot simultaneously be $0$ and exchanging their roles leads to an isomorphic situation. Furthermore, if necessary we can replace $(k,l)$ by $(-k,-l)$. Now looking at $\{(m,n)\mid m,n \in \mathbb Z, 0 \leq m < k\}$ as a system of representatives it is straightforward to check that the resulting graph is isomorphic to $G_{k,l}$.
\end{proof}

Note that every $G_{k,l}$ occurs as a Cayley graph, more precisely, it is the Cayley graph of the group $\Gamma_{k,l} :=\mathbb Z \oplus \mathbb Z_{\gcd(k,l)}$ with generators 
\[
\rgen := \left(\frac{l}{\gcd(k,l)},1\right)  \quad \text{ and }  \quad \ugen := \left(-\frac{k}{\gcd(k,l)},1\right).
\]
For $\gcd(k,l) = 1$ we note that $\mathbb Z \oplus \mathbb Z_1 = \mathbb Z$ and consistently with the above, $G_{k,l}$ is the Cayley graph of the group $\Gamma_{k,l} = \mathbb Z$ with generators $\rgen := l$ and $\ugen := -k$.

Using this representation of $G_{k,l}$, vertical edges correspond to the generator $\ugen$ and horizontal edges correspond to the generator $\rgen$. For $0 \leq m < k$ the vertex $(m,n)$ of $G_{k,l}$ corresponds to the group element $\rgen^m\ugen^n$; we point out that if we refer to a vertex of $G_{k,l}$ or group element of $\Gamma_{k,l}$ as a pair $(m,n)$, we always interpret it as $\rgen^m\ugen^n$, and never as $(m \in \mathbb Z,n\in \mathbb Z_{\gcd(k,l)}) \in \mathbb Z \times \mathbb Z_{\gcd(k,l)}$.

Note that we do not need to consider the group $\mathbb Z \oplus V_4$ since $G_{4,0}$ also occurs as a Cayley graph of $\Gamma_{4,0} = \mathbb Z \oplus \mathbb Z_4$. Further note that (by replacing generators by their inverses and swapping their roles) we have that $G_{k,l} \simeq G_{k,-l}$, $G_{k,l} \simeq G_{l,k}$ for $l > 0$, and $G_{k,l} \simeq G_{-l,-k}$ for $l < 0$.

We can specify a walk in $G_{k,l}$ by giving the starting vertex together with a series of generators and their inverses, denoted by $\lgen := \rgen ^{-1}$ and $\dgen = \ugen^{-1}$. To avoid confusion with the group element obtained by multiplication of these generators, we will put the generators defining the walk in square brackets. For a more compact representation we will also represent repeated patterns by exponentiation. For example, the following expressions all define the same walk in $G_{3,1}$, see Figure \ref{f:Walk-3-1}:
\begin{align*}
    P&=(0,0)(1,0)(1,1)(2,1)(2,2)(0,1)(0,2)(0,3)\\
    &=(0,0)[\rgen \ugen \rgen \ugen \rgen \ugen \ugen]\\
    &=(0,0)[\rgen][\ugen\rgen]^2 [\ugen \ugen]\\
    &=(0,0)[\rgen \ugen]^2[ \rgen ][\ugen]^2.
\end{align*}

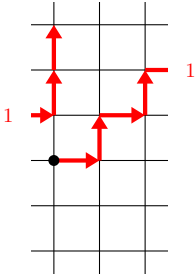
\begin{figure}[ht!]
\centering

\begin{tikzpicture}[scale = .6, hamray/.style={ultra thick,red},hamray3/.style={ultra thick,black!30!green},hamray2/.style={ultra thick,blue}]

\begin{scope}

\path[clip] (-.5,-2.5) rectangle (2.5,3.5);

%draw grid
\foreach \i in {0,...,2}
    \draw (\i,-3)--(\i,6);
\foreach \i in {-2,...,5}
    \draw (-1,\i)--(3,\i);

%draw hamilton double rays

    \draw [-{Triangle[scale=.7]},hamray] (0,0)--(1,0);
    \draw [-{Triangle[scale=.7]},hamray] (1,0)--(1,1); 
    \draw [-{Triangle[scale=.7]},hamray] (1,1)--(2,1); 
    \draw [-{Triangle[scale=.7]},hamray] (2,1)--(2,2);
    \draw [-{Triangle[scale=.7]},hamray] (2,2)--(3,2); 
    \draw [-{Triangle[scale=.7]},hamray](-1,1)--(0,1) ;
    \draw [-{Triangle[scale=.7]},hamray](0,1)--(0,2);
    \draw [-{Triangle[scale=.7]},hamray] (0,2)--(0,3);
\end{scope}

%nodes that show which boundary points are identified
\node[circle, fill, scale=0.4] at (0,0) {};
\node[red,scale=0.7] at (-1,1) {1};
\node[red,scale=0.7] at (3,2) {1};

\end{tikzpicture}

\caption{The walk $P$ in $G_{3,1}$ starting at the black dot at $(0,0)$. The label indicates where the edge from $(2,2)$ to $(0,1)$ leaves and enters the diagram.}\label{f:Walk-3-1}
\end{figure}

\section{Hamiltonian decompositions of 4-regular Cayley graphs}\label{s:proofs}

A \emph{vertical cut} of $G_{k,l}$ is the orbit of a horizontal edge under the action of the subgroup generated by $\ugen$, or in other words, the set all vertical translates of a horizontal edge. A \emph{horizontal cut} of $G_{k,l}$ is the orbit of a vertical edge under the action of the subgroup generated by $\rgen$. Let $\tilde E$ be a subset of the edges of $G_{k,l}$. We say that $\tilde E$ \emph{prevails} in a vertical (horizontal) cut, if for every edge $e$ in this cut there are $a > 0$ and $b > 0$ such that $\ugen 
^a e$ and $\dgen^b e$ ($\rgen 
^a e$ and $\lgen^b e$) lie in $\tilde E$.
We say that $\tilde E$ is \emph{horizontally (vertically) prevalent} if there is a horizontal (vertical) cut in which $\tilde E$ prevails and \emph{bi-prevalent} if it is both horizontally and vertically prevalent. We say that a decomposition $E_1 \uplus E_2$ \emph{horizontally prevalent} if both $E_1$ and $E_2$ are horizontally prevalent, and similarly for vertically and bi-prevalent.

\begin{lem}
\label{lem:inductionstep}
\begin{enumerate}
    \item If $G_{k,l}$ admits a vertically prevalent or bi-prevalent decomposition into Hamiltonian double-rays, then  so does $G_{k+2,l}$.
    \item If $l > 0$ and $G_{k,l}$ admits a horizontally prevalent or bi-prevalent decomposition into Hamiltonian double-rays, then  so does $G_{k,l+2}$.
    \item Analogous statements hold for decompositions into Hamiltonian circles, and for decompositions into a Hamiltonian double-ray and a Hamiltonian circle.
\end{enumerate}
\end{lem}

\begin{proof}
For the proof of the first statement let $E_1 \uplus E_2$ be a vertically prevalent decomposition of $G_{k,l}$ into Hamiltonian double-rays, and let $C$ be a vertical cut in which both $E_1$ and $E_2$ prevail. Without loss of generality, $C$ consists of all edges connecting $(k-1,n)$ to $(0,n-l)$ for $n \in \mathbb Z$; this can always be achieved by applying an appropriate automorphism. We write $e_n$ for the edge connecting $(k-1,n)$ to $(0,n-l)$. If $e_n \in E_1$, then we define $h_n = \min\{h > 0\mid e_{n+h} \in E_1\}$. Similarly, if $e_n \in E_2$, then we define $h_n = \min\{h > 0\mid e_{n+h} \in E_2\}$.

Next note that $G_{k+2,l}$ can be obtained from $G_{k,l}$ by the following procedure: remove all edges in $C$, and for every $j \in \mathbb Z$ add vertices $(k,j)$ and $(k+1,j)$ and the appropriate edges. Using this construction of $G_{k+2,l}$ we transform $E_1$ and $E_2$ into subsets $E_1'$ and $E_2'$ of the edge set of $G_{k+2,l}$ as follows.
The set $E_i'$ consists of $E_i \setminus C$ and the edges of the walks
\[
W_n := (k-1,n)[\rgen][\ugen]^{h_n-1}[\rgen][\dgen]
^{h_n-1}[\rgen]
\]
for every $n$ with $e_n \in E_i$. See Figure \ref{fig:4-2and6-2} for an example in the case of $G_{4,2}$.

\begin{figure}[ht!]
\centering
\begin{minipage}{.5\textwidth}
\centering
\vspace{1,2cm}
\begin{tikzpicture}[scale = .4, hamray/.style={ultra thick,red},hamray2/.style={ultra thick,blue}]

\begin{scope}

\path[clip] (-.5,-6.5) rectangle (3.5,5.5);

%draw grid
\foreach \i in {0,...,3}
    \draw (\i,-7)--(\i,6);
\foreach \i in {-6,...,5}
    \draw (-1,\i)--(9,\i);

%shortcuts for up/down/right (for some reason \U already exists, but I think I only overwrite it in this scope)
\newcommand{\U}{--++(0,1)}
\newcommand{\R}{--++(1,0)}
\newcommand{\D}{--++(0,-1)}

%draw hamilton double rays

    \draw [hamray]{
        (-1,-5)\R\R\R\U\U\U\R\R
         (-1,-4)\R\R\U\U\U\R\R\R
         (-1,-3)\R\U\U\U\R\R\R\U\U\U\R
          (-1,1)\R\R\R\U\U\U\R\R
         
         };
\end{scope}

%nodes that show which boundary points are identified
\node[red,scale=0.7] at (4,3) {1};
\node[red,scale=0.7] at (-1,1) {1};
\node[red,scale=0.7] at (4,-1) {2};
\node[red,scale=0.7] at (-1,-3) {2};
\node[red,scale=0.7] at (4,-2) {3};
\node[red,scale=0.7] at (-1,-4) {3};

\end{tikzpicture}
\end{minipage}%
\begin{minipage}{.5\textwidth}
\centering
\begin{tikzpicture}[scale = .4, hamray/.style={ultra thick,red},hamray2/.style={ultra thick,blue}]

\begin{scope}

\path[clip] (-.5,-6.5) rectangle (5.5,8.5);

%draw grid
\foreach \i in {0,...,5}
    \draw (\i,-7)--(\i,9);
\foreach \i in {-6,...,8}
    \draw (-1,\i)--(6,\i);

%shortcuts for up/down/right (for some reason \U already exists, but I think I only overwrite it in this scope)
\newcommand{\U}{--++(0,1)}
\newcommand{\R}{--++(1,0)}
\newcommand{\D}{--++(0,-1)}

%draw hamilton double rays

    \draw [hamray]{
        (-1,-5)\R\R\R\U\U\U\R\R\R\R
         (-1,-4)\R\R\U\U\U\R\R\R\U\U\U\R\D\D\D\R
         (-1,-3)\R\U\U\U\R\R\R\U\U\U\R\R\R
          (-1,1)\R\R\R\U\U\U\R\R\U\U\U\R\D\D\D\R
         
         };
\end{scope}

%nodes that show which boundary points are identified
\node[red,scale=0.7] at (6,3) {1};
\node[red,scale=0.7] at (-1,1) {1};
\node[red,scale=0.7] at (6,-1) {2};
\node[red,scale=0.7] at (-1,-3) {2};
\node[red,scale=0.7] at (6,-2) {3};
\node[red,scale=0.7] at (-1,-4) {3};

\end{tikzpicture}
\end{minipage}%
\caption{Extending a bi-prevalent decomposition of $G_{4,2}$ into Hamiltonian double-rays to a bi-prevalent decomposition of $G_{6,2}$.}\label{fig:4-2and6-2}
\end{figure}
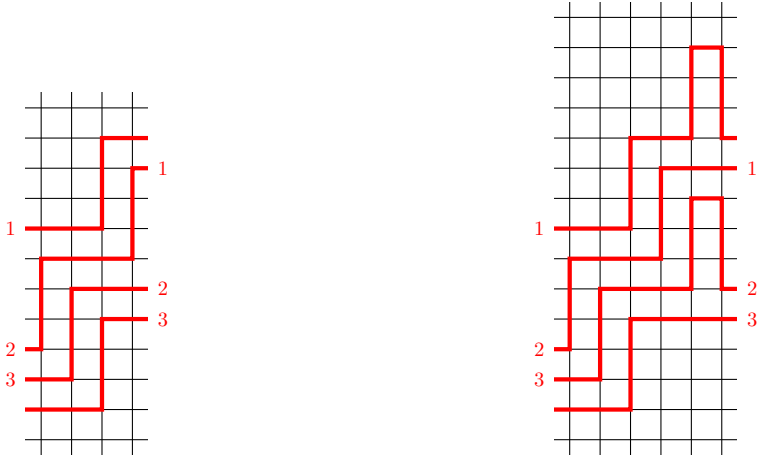

Vertical prevalence of $E_1$ and $E_2$ ensures that $h(e)$ is finite, thus $W_n$ is a finite walk. Note that $W_n$ starts in $(k-1,n)$, ends in $(0,n-l)$, and additionally contains the vertices $(k,j)$ and $(k+1,j)$ for $n \leq j < n+h_n$. In particular, by definition of $n_h$, the paths in $\{W_n\mid e_n \in E_i \cap C\}$ are vertex disjoint and their union covers the vertices $(k,j)$ and $(k+1,j)$ for $j \in \mathbb Z$.

We now show that $E_1'$ and $E_2'$ form the desired decomposition of $G_{k+2,l}$. The graph spanned by $E_i'$ is obtained from the graph spanned by $E_i$ by replacing edges in $C$ by disjoint paths with the same endpoints. Since the graph spanned by $E_i$ was connected and $2$-regular, the same is true for the graph spanned by $E_i'$. It contains all vertices $(i,j)$ for $i < k$ and $j \in \mathbb Z$ since $E_i$ was spanning, and it contains all vertices $(k,j)$ and $(k+1,j)$ due to the above observation.

To see that $E_1'$ and $E_2'$ are disjoint, first note that $E_1 \setminus C$ and $E_2 \setminus C$ are disjoint, so we only need to show that the walks $W_n$ are edge disjoint. 
Take $e_n \in E_1$ and $e_m \in E_2$. If $W_n$ and $W_m$ intersect in a horizontal edge, then either $n = m$ (for the first and last edge), or $n+h_n = m + h_m$ (for the central edge). This is not possible because $e_n$ and $e_{n+h_n}$ are in $E_1$ whereas $e_m$ and $e_{m+h_m}$ are in $E_2$. For vertical edges note that if $W_n$ contains an edge from $(k,j)$ to $(k,j+1)$ or from $(k+1,j+1)$ to $(k+1,j)$, then $e_{j+1} \in E_2$. Similarly, if $W_m$ contains such an edge, then $e_{j+1} \in E_1$. This implies that they cannot contain the same vertical edge, so $W_n$ and $W_m$ must be disjoint.

The decomposition $E_1' \uplus E_2'$ is vertically prevalent since $E_i'$ contains the edge from $(k+1,n)$ to $(0,n-l)$ if and only if $E_i$ contains an edge from $(k-1,n)$ to $(0,n-l)$. If we additionally assume that $E_1 \uplus E_2$ is horizontally prevalent, then so is $E_1' \uplus E_2'$ since any horizontal cut in $G_{k,l}$ is fully contained in a horizontal cut in $G_{k+2,l}$.

This finishes the proof of the first statement. The second statement follows from the fact that for $l > 0$ there is an isomorphism between $G_{k,l}$ and $G_{l,k}$ which swaps horizontal and vertical cuts. The third statement can be proved in a completely analogous fashion (with the additional observation that $E_i$ having tails in different ends of $G_{k,l}$ implies that $E_i'$ has tails in different ends of $G_{k+2,l}$), we leave the details to the reader.
\end{proof}

\begin{lem}
\label{lem:liftcycle}
Let $\Gamma$ be a $2$-ended abelian group\footnote{We note that the number of ends of a Cayley graph does not depend on the generating set chosen, see for example \cite{DM18}
%, and so it makes sense to talk about the number of ends of a group
.},
let $S$ be a generating set, and let $\Delta$ be an infinite cyclic subgroup of $\Gamma$ generated by $a$. Let $G$ be the Cayley graph of $\Gamma$ with respect to $S$ and let $H$ be the Cayley graph of $\Gamma/\Delta$ with respect to the generating set $S \Delta$, where we allow multiple edges in case $s_1 \Delta = s_2\Delta$ for $s_1,s_2 \in S$, and let $\pi \colon G \to H$ be the projection map. Let $C$ be a Hamiltonian cycle in $H$ and let $k$ be the sum of the generators used along this cycle.

\begin{enumerate}
    \item If $k = a$, then $\pi^{-1}(C)$ is a Hamiltonian double-ray in $G$.
    \item If $k = a^2$, then $\pi^{-1}(C)$ is a Hamiltonian circle in $G$.
\end{enumerate}
\end{lem}

\begin{proof}
Every vertex  $v$ of $G$ has exactly two incident edges in the preimage (namely the two edges corresponding to the same generators as the edges in $C$ incident to $\pi(v)$). Hence $\pi^{-1}(C)$ is 2-regular.

Since $C$ is spanning in $H$ we know that every component of $\pi^{-1}(C)$ contains elements of all cosets with respect to $\Delta$. Moreover, if $k = a$, then by following the edges corresponding to the same generators as edges along $C$, we see that for any element $v$ of $\Gamma$ the element $va$ lies in the same component of $\pi^{-1}(C)$ as $v$. Hence in this case $\pi^{-1}(C)$ is connected and thus a Hamiltonian double-ray.

If $k = a^2$, then a similar argument shows that $v$ and $va^i$ lie in the same component if and only if $i$ is even, so $\pi^{-1}(C)$ has exactly two components. Each of the components is invariant under the action of $a^2$ and since high positive and negative powers of $a$ converge to different ends of $G$ we conclude that each component contains tails in both ends. Thus $\pi^{-1}(C)$ is a Hamiltonian circle.
\end{proof}

\begin{lem}
\label{lem:inductionbase}
\begin{enumerate}
    \item If $3 \leq k = l+2$, then $G_{k,l}$ has a bi-prevalent decomposition into Hamiltonian double-rays.
    \item If $3 \leq k = l+3$, then $G_{k,l}$ has a bi-prevalent decomposition into one Hamiltonian double-ray and one Hamiltonian circle.
    \item If $3 \leq k = l+4$, then $G_{k,l}$ has a bi-prevalent decomposition into Hamiltonian circles.
\end{enumerate}
\end{lem}
\begin{proof}
Let $A$ be the group generated by $\agen = \lgen\ugen$. Note that the quotient group is a cyclic group generated by $\rgen A = \ugen A$. The Cayley graph $H$ of this quotient group is a cycle where each edge has been replaced by two parallel edges (corresponding to the two different generators. 

If $k = l+2$, pick a Hamiltonian cycle $C$ of $H$ using exactly $k-1$ edges corresponding to the generator $\rgen$. Recall that $(\rgen^k\ugen^l) = \id$, so $C$ contains $l+1 = k-1$ edges corresponding to the generator $\ugen$ and the same is true for the (edge-)complement of $C$. Note that $(\rgen^{k-1}\ugen^{l+1}) = \agen$. Thus Lemma \ref{lem:liftcycle} implies that the preimages $E_1$ of $C$ and $E_2$ of its complement under the natural projection map form a decomposition of $G_{k,l}$ into Hamiltonian double-rays. See the left picture in Figure \ref{fig:4-2and4-0} for an example in $G_{4,2}$.

The argument for the case $k = l+4$ is completely analogous, but with $k-2 = l+2$ edges corresponding to generators $\ugen$ and $\rgen$ respectively; note that $(\rgen^{k-2}\ugen^{l+2}) = \agen^2$. See the right picture in Figure \ref{fig:4-2and4-0} for an example in $G_{4,0}$.

For $k = l+3$ we choose the cycle $C$ with $k-1$ edges corresponding to the generator $\rgen$ and (consequently) $l+1 = k-2$ edges corresponding to the generator $\ugen$. Thus the complement will contain  $k-2$ edges corresponding to the generator $\rgen$ and $l+2$ edges corresponding to the generator $\ugen$. It follows that the preimage $E_1$ of $C$ is a Hamiltonian double-ray and the preimage $E_2$ of its complement is a Hamiltonian circle. See Figure \ref{fig:4-1} for an example in $G_{4,1}$.

It remains to show that the decompositions are bi-prevalent. If we follow the edges of $E_1$ starting at $\id$ then the first edge corresponding to $\ugen$ lies in the same horizontal cut $H$ as the edge from $\id$ to $\ugen$, and the first edge corresponding to $\rgen$ lies in the same vertical cut $K$ as the edge from $\id$ to $\rgen$. Since $C$ contains edges of both types, we know that $E_1$ contains at least one edge in both of these cuts. Similarly, $E_2$ contains edges in both of these cuts.

Note that $\agen^k = \rgen ^k \ugen ^k = \ugen^{k-l}$. Since the decomposition $E_1 \uplus E_2$ is invariant under the action of $A$, this implies that for every edge $e \in E_i \cap K$ the edges $\ugen^{i(k-l)}e$ for $i \in \mathbb Z$ are also in $E_i \cap K$; thus the decomposition is vertically prevalent. If $l =0$, then the horizontal cut $H$ is finite, and the decomposition is horizontally prevalent because both parts intersect with $H$, otherwise the same argument as for $K$ applies.

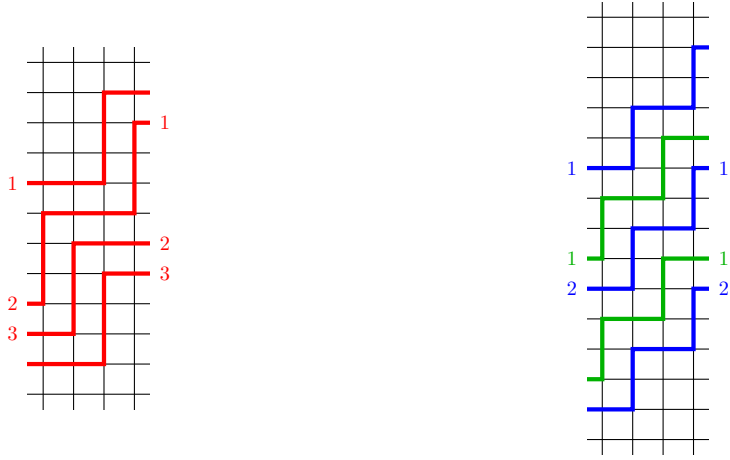
\begin{figure}[ht!]
\centering
\begin{minipage}{.5\textwidth}
\centering
\begin{tikzpicture}[scale = .4, hamray/.style={ultra thick,red},hamray2/.style={ultra thick,blue}]

\begin{scope}

\path[clip] (-.5,-6.5) rectangle (3.5,5.5);

%draw grid
\foreach \i in {0,...,3}
    \draw (\i,-7)--(\i,6);
\foreach \i in {-6,...,5}
    \draw (-1,\i)--(9,\i);

%shortcuts for up/down/right (for some reason \U already exists, but I think I only overwrite it in this scope)
\newcommand{\U}{--++(0,1)}
\newcommand{\R}{--++(1,0)}
\newcommand{\D}{--++(0,-1)}

%draw hamilton double rays

    \draw [hamray]{
        (-1,-5)\R\R\R\U\U\U\R\R
         (-1,-4)\R\R\U\U\U\R\R\R
         (-1,-3)\R\U\U\U\R\R\R\U\U\U\R
          (-1,1)\R\R\R\U\U\U\R\R
         
         };
\end{scope}

%nodes that show which boundary points are identified
\node[red,scale=0.7] at (4,3) {1};
\node[red,scale=0.7] at (-1,1) {1};
\node[red,scale=0.7] at (4,-1) {2};
\node[red,scale=0.7] at (-1,-3) {2};
\node[red,scale=0.7] at (4,-2) {3};
\node[red,scale=0.7] at (-1,-4) {3};

\end{tikzpicture}
\end{minipage}%
\begin{minipage}{.5\textwidth}
\centering
\begin{tikzpicture}[scale = .4, hamray/.style={ultra thick,red},hamray3/.style={ultra thick,black!30!green},hamray2/.style={ultra thick,blue}]

\begin{scope}

\path[clip] (-.5,-6.5) rectangle (3.5,8.5);

%draw grid
\foreach \i in {0,...,3}
    \draw (\i,-7)--(\i,9);
\foreach \i in {-6,...,8}
    \draw (-1,\i)--(9,\i);

%shortcuts for up/down/right (for some reason \U already exists, but I think I only overwrite it in this scope)
\newcommand{\U}{--++(0,1)}
\newcommand{\R}{--++(1,0)}
\newcommand{\D}{--++(0,-1)}

%draw hamilton double rays

    \draw [hamray2]{
        (-1,-5)\R\R\U\U\R\R\U\U\R
          (-1,-1)\R\R\U\U\R\R\U\U\R
          (-1,3)\R\R\U\U\R\R\U\U\R
         };
         \draw [hamray3]{
        (-1,-4)\R\U\U\R\R\U\U\R\R
          (-1,-0)\R\U\U\R\R\U\U\R\R
         };
\end{scope}

%nodes that show which boundary points are identified
\node[blue,scale=0.7] at (4,3) {1};
\node[blue,scale=0.7] at (-1,3) {1};
\node[blue,scale=0.7] at (4,-1) {2};
\node[blue,scale=0.7] at (-1,-1) {2};
\node[black!30!green,scale=0.7] at (4,0) {1};
\node[black!30!green,scale=0.7] at (-1,0) {1};

\end{tikzpicture}
\end{minipage}%
\caption{A bi-prevalent double-ray in $G_{4,2}$ whose complement is a bi-prevalent double-ray and a bi-prevalent circle in $G_{4,0}$ whose complement is a bi-prevalent circle.}\label{fig:4-2and4-0}
\end{figure}

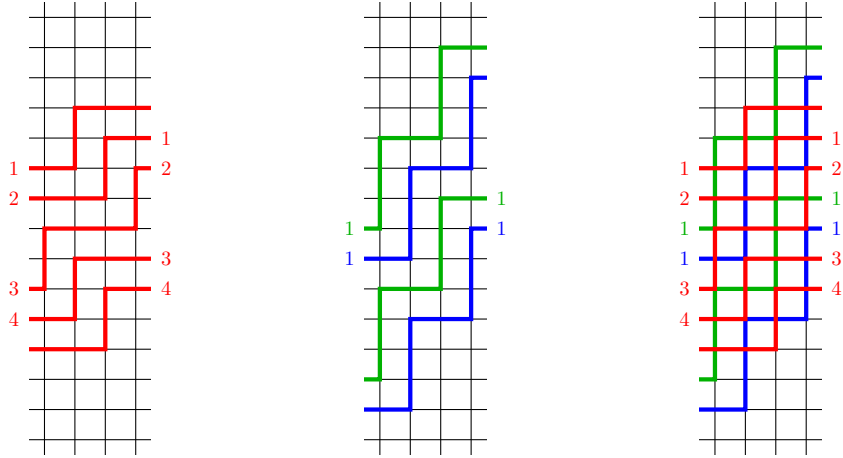
\begin{figure}
\centering
\begin{minipage}{.3\textwidth}
\centering
\begin{tikzpicture}[scale = .4, hamray/.style={ultra thick,red},hamray3/.style={ultra thick,black!30!green},hamray2/.style={ultra thick,blue}]

\begin{scope}

\path[clip] (-.5,-6.5) rectangle (3.5,8.5);

%draw grid
\foreach \i in {0,...,3}
    \draw (\i,-7)--(\i,9);
\foreach \i in {-6,...,8}
    \draw (-1,\i)--(9,\i);

%shortcuts for up/down/right (for some reason \U already exists, but I think I only overwrite it in this scope)
\newcommand{\U}{--++(0,1)}
\newcommand{\R}{--++(1,0)}
\newcommand{\D}{--++(0,-1)}

%draw hamilton double rays
         \draw [hamray]{
        (-1,-3)\R\R\R\U\U\R\R
        (-1,-2)\R\R\U\U\R\R\R
        (-1,-1)\R\U\U\R\R\R\U\U\R
        (-1,2)\R\R\R\U\U\R\R
        (-1,3)\R\R\U\U\R\R\R
         };
\end{scope}

%nodes that show which boundary points are identified
\node[red,scale=0.7] at (4,4) {1};
\node[red,scale=0.7] at (-1,3) {1};
\node[red,scale=0.7] at (4,3) {2};
\node[red,scale=0.7] at (-1,2) {2};
\node[red,scale=0.7] at (4,0) {3};
\node[red,scale=0.7] at (-1,-1) {3};
\node[red,scale=0.7] at (4,-1) {4};
\node[red,scale=0.7] at (-1,-2) {4};

\end{tikzpicture}
\end{minipage}%
\begin{minipage}{.3\textwidth}
\centering
\begin{tikzpicture}[scale = .4, hamray/.style={ultra thick,red},hamray3/.style={ultra thick,black!30!green},hamray2/.style={ultra thick,blue}]

\begin{scope}

\path[clip] (-.5,-6.5) rectangle (3.5,8.5);

%draw grid
\foreach \i in {0,...,3}
    \draw (\i,-7)--(\i,9);
\foreach \i in {-6,...,8}
    \draw (-1,\i)--(9,\i);

%shortcuts for up/down/right (for some reason \U already exists, but I think I only overwrite it in this scope)
\newcommand{\U}{--++(0,1)}
\newcommand{\R}{--++(1,0)}
\newcommand{\D}{--++(0,-1)}

%draw hamilton double rays

    \draw [hamray2]{
        (-1,-5)\R\R\U\U\U\R\R\U\U\U\R
        (-1,0)\R\R\U\U\U\R\R\U\U\U\R
         };
        \draw [hamray3]{
        (-1,-4)\R\U\U\U\R\R\U\U\U\R\R
          (-1,1)\R\U\U\U\R\R\U\U\U\R\R
         };
\end{scope}

%nodes that show which boundary points are identified
\node[blue,scale=0.7] at (4,1) {1};
\node[blue,scale=0.7] at (-1,0) {1};
\node[black!30!green,scale=0.7] at (4,2) {1};
\node[black!30!green,scale=0.7] at (-1,1) {1};

\end{tikzpicture}
\end{minipage}%
\begin{minipage}{.3\textwidth}
\centering
\begin{tikzpicture}[scale = .4, hamray/.style={ultra thick,red},hamray3/.style={ultra thick,black!30!green},hamray2/.style={ultra thick,blue}]

\begin{scope}

\path[clip] (-.5,-6.5) rectangle (3.5,8.5);

%draw grid
\foreach \i in {0,...,3}
    \draw (\i,-7)--(\i,9);
\foreach \i in {-6,...,8}
    \draw (-1,\i)--(9,\i);

%shortcuts for up/down/right (for some reason \U already exists, but I think I only overwrite it in this scope)
\newcommand{\U}{--++(0,1)}
\newcommand{\R}{--++(1,0)}
\newcommand{\D}{--++(0,-1)}

%draw hamilton double rays

    \draw [hamray2]{
        (-1,-5)\R\R\U\U\U\R\R\U\U\U\R
        (-1,0)\R\R\U\U\U\R\R\U\U\U\R
         };
        \draw [hamray3]{
        (-1,-4)\R\U\U\U\R\R\U\U\U\R\R
          (-1,1)\R\U\U\U\R\R\U\U\U\R\R
         };
         \draw [hamray]{
        (-1,-3)\R\R\R\U\U\R\R
        (-1,-2)\R\R\U\U\R\R\R
        (-1,-1)\R\U\U\R\R\R\U\U\R
        (-1,2)\R\R\R\U\U\R\R
        (-1,3)\R\R\U\U\R\R\R
         };
\end{scope}

%nodes that show which boundary points are identified
\node[blue,scale=0.7] at (4,1) {1};
\node[blue,scale=0.7] at (-1,0) {1};
\node[black!30!green,scale=0.7] at (4,2) {1};
\node[black!30!green,scale=0.7] at (-1,1) {1};
\node[red,scale=0.7] at (4,4) {1};
\node[red,scale=0.7] at (-1,3) {1};
\node[red,scale=0.7] at (4,3) {2};
\node[red,scale=0.7] at (-1,2) {2};
\node[red,scale=0.7] at (4,0) {3};
\node[red,scale=0.7] at (-1,-1) {3};
\node[red,scale=0.7] at (4,-1) {4};
\node[red,scale=0.7] at (-1,-2) {4};

\end{tikzpicture}
\end{minipage}%
\caption{A bi-prevalent double-ray in $G_{4,1}$ whose complement is a bi-prevalent circle.}\label{fig:4-1}
\end{figure}

\end{proof}

Note that the restriction $k \geq 3$ in the first condition is necessary since $G_{1,-1}$ and $G_{2,0}$ are not $4$-regular. In the other two cases it is merely required to enable us to apply Lemma \ref{lem:liftcycle}; if $k$ was smaller than $3$ in these cases then the quotient group would become trivial or infinite.

By Proposition \ref{prp:structure} the only $4$-regular Cayley graphs of abelian groups are either the square grid, or of the form $G_{k,l}$. Since the square grid satisfies the conclusion of Theorem \ref{thm:main}, it will suffice to show that every $4$-regular $G_{k,l}$ which satisfies \ref{itm:parity} has a decomposition into Hamiltonian double-rays and a decomposition into Hamiltonian circles. It is easy to check that \ref{itm:parity} is satisfied if and only if $k$ and $l$ have the same parity. Hence the following lemma completes the proof of Theorems \ref{thm:main}, \ref{thm:main-circle}, and \ref{thm:split}.

\begin{lem}
\begin{enumerate} Let $k \in \mathbb{N}$ and $l \in \mathbb{Z}$ be such that $G_{k,l}$ is $4$-regular.
    \item If $k \equiv l \mod 2$, then $G_{k,l}$ has a decomposition into two Hamiltonian double-rays.
    \item If $k \equiv l \mod 2$, then $G_{k,l}$ has a decomposition into two Hamiltonian circles.
    \item If $k \not \equiv l  \mod 2$, then $G_{k,l}$ has a decomposition into a Hamiltonian double-ray and a Hamiltonian circle.
\end{enumerate}
\end{lem}

\begin{proof}
By the remark after Proposition \ref{p:classification} it is sufficient to prove the lemma for $k \geq l \geq 0$. Moreover, we can ignore the cases $k=1,l \leq 1$ and $k=2,l=0$ since they do not lead to $4$-regular graphs.

For the first part, note that by Lemma \ref{lem:inductionbase}, there are bi-prevalent decompositions of $G_{4,2}$ and $G_{3,1}$ into Hamiltonian double-rays. Lemma \ref{lem:inductionstep} and induction finish the proof apart from the cases $k=l=2$ and $k > l = 0$. 
For $l = 0$ it is enough to show that $G_{4,0}$ has a vertically prevalent decomposition into Hamiltonian double-rays, once this is done we can induct using the first part of Lemma \ref{lem:inductionstep}.
For $k=l=2$ we explicitly construct a decomposition. Both of these are presented in Figure \ref{fig:4-0and2-2}.

In order to prove the second statement, recall that there is an isomorphism between $G_{k,l}$ and $G_{k,-l}$ and note that this isomorphism preserves horizontal and vertical cuts. Hence by Lemma \ref{lem:inductionbase} the graphs $G_{4,0}$
and $G_{3,1} \simeq G_{3,-1}$ have bi-prevalent decompositions into Hamiltonian circles. Furthermore we give in Figure \ref{fig:4-2} a bi-prevalent decomposition of $G_{4,2}$ into Hamiltonian circles. An inductive application of Lemma \ref{lem:inductionstep} then finishes the proof of the second part apart from the case $k =l=2$. To see that $G_{2,2}$ has a decomposition into Hamiltonian circles, simply consider the decomposition into horizontal and vertical edges.

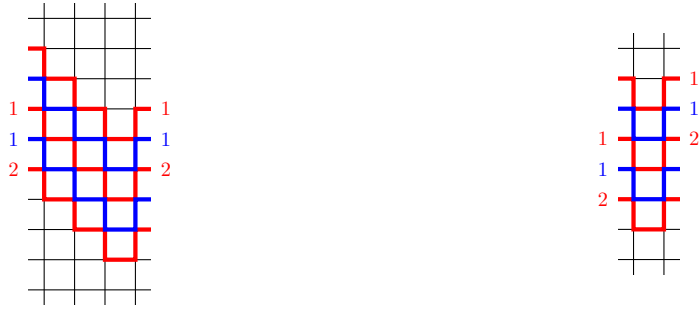
\begin{figure}[ht!]
\centering
\begin{minipage}{.5\textwidth}
\centering
\begin{tikzpicture}[scale = .4, hamray/.style={ultra thick,red},hamray2/.style={ultra thick,blue}]

\begin{scope}

\path[clip] (-.5,-5.5) rectangle (3.5,4.5);

%draw grid
\foreach \i in {0,...,3}
    \draw (\i,-6)--(\i,5);
\foreach \i in {-5,...,4}
    \draw (-1,\i)--(9,\i);

%shortcuts for up/down/right (for some reason \U already exists, but I think I only overwrite it in this scope)
\newcommand{\U}{--++(0,1)}
\newcommand{\R}{--++(1,0)}
\newcommand{\D}{--++(0,-1)}

%draw hamilton double rays
\foreach \i in {-1,1,3}
{
    \draw [hamray]{
        (-1,\i)\R
        \foreach \i in {1,2}{\D\R}
        \foreach \i in {1}{\D\R\U\R}
    };
}
\foreach \i in {0,2}
{
    \draw [hamray2]{
        (-1,\i)\R
        \foreach \i in {1,2}{\D\R}
        \foreach \i in {1}{\D\R\U\R}
    };
}
\end{scope}

%nodes that show which boundary points are identified
\node[red,scale=0.7] at (4,1) {1};
\node[red,scale=0.7] at (-1,1) {1};
\node[red,scale=0.7] at (4,-1) {2};
\node[red,scale=0.7] at (-1,-1) {2};
\node[blue,scale=0.7] at (4,-0) {1};
\node[blue,scale=0.7] at (-1,-0) {1};

\end{tikzpicture}
\end{minipage}%
\begin{minipage}{.5\textwidth}
\centering
\begin{tikzpicture}[scale = .4, hamray/.style={ultra thick,red},hamray2/.style={ultra thick,blue}]

\begin{scope}

\path[clip] (-.5,-3.5) rectangle (1.5,4.5);

%draw grid
\foreach \i in {0,...,1}
    \draw (\i,-4)--(\i,5);
\foreach \i in {-3,...,4}
    \draw (-1,\i)--(4,\i);

%shortcuts for up/down/right (for some reason \U already exists, but I think I only overwrite it in this scope)
\newcommand{\U}{--++(0,1)}
\newcommand{\R}{--++(1,0)}
\newcommand{\D}{--++(0,-1)}

%draw hamilton double rays
\foreach \i in {-1,1,3}
{
    \draw [hamray]{
        (-1,\i)\R\D\R\U\R
    };
}
\foreach \i in {0,2}
{
    \draw [hamray2]{
        (-1,\i)\R\D\R\U\R
    };
}
\end{scope}

%nodes that show which boundary points are identified
\node[red,scale=0.7] at (2,3) {1};
\node[red,scale=0.7] at (-1,1) {1};
\node[red,scale=0.7] at (2,1) {2};
\node[red,scale=0.7] at (-1,-1) {2};
\node[blue,scale=0.7] at (2,2) {1};
\node[blue,scale=0.7] at (-1,-0) {1};

\end{tikzpicture}
\end{minipage}%
\caption{A bi-prevalent decomposition of $G_{4,0}$ and a vertically prevalent decomposition of $G_{2,2}$ into Hamiltonian double-rays.}\label{fig:4-0and2-2}
\end{figure}

\begin{figure}
\centering
\begin{tikzpicture}[scale = .4, hamray/.style={ultra thick,red},hamray3/.style={ultra thick,black!30!green},hamray2/.style={ultra thick,blue}]

\begin{scope}

\path[clip] (-.5,-5.5) rectangle (3.5,6.5);

%draw grid
\foreach \i in {0,...,3}
    \draw (\i,-6)--(\i,7);
\foreach \i in {-5,...,6}
    \draw (-1,\i)--(7,\i);

%shortcuts for up/down/right (for some reason \U already exists, but I think I only overwrite it in this scope)
\newcommand{\U}{--++(0,1)}
\newcommand{\R}{--++(1,0)}
\newcommand{\D}{--++(0,-1)}

%draw hamilton double rays

    \draw [hamray2]{
        (-1,5)\R\D\R\D\R\D\R\U\R
        (-1,1)\R\D\R\D\R\D\R\U\R
         };
        \draw [hamray3]{
        (-1,3)\R\D\R\D\R\D\R\U\R
          (-1,-1)\R\D\R\D\R\D\R\U\R
         };
\end{scope}

%nodes that show which boundary points are identified
\node[blue,scale=0.7] at (4,3) {1};
\node[blue,scale=0.7] at (-1,1) {1};
\node[black!30!green,scale=0.7] at (4,1) {1};
\node[black!30!green,scale=0.7] at (-1,-1) {1};

\end{tikzpicture}
\caption{A bi-prevalent decomposition of $G_{4,2}$ into Hamiltonian circles (note that the complement of this circle is a vertical translation of it).}
\label{fig:4-2}
\end{figure}
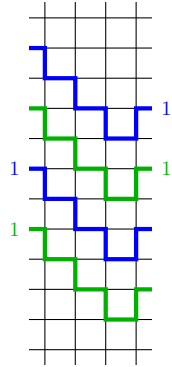

For the proof of the third part note that by Lemma \ref{lem:inductionbase} the graphs $G_{5,2}$, $G_{4,1}$, and $G_{3,0}$ have bi-prevalent decompositions into a Hamiltonian double-ray and a Hamiltonian circle. Inductive application of Lemma \ref{lem:inductionstep} finishes the proof apart from the case $k = 2, l=1$. However, the graph $G_{2,1}$ has a decomposition into a Hamiltonian double-ray and a Hamiltonian circle as well; again consider the decomposition into vertical and horizontal edges.
\end{proof}

\section{A generalisation of Alspach's conjecture}
The conditions in Alspach's conjecture arise quite naturally; since every Hamiltonian cycle must meet every cut of a graph in an even number of edges, for a Hamiltonian decomposition to exist each cut must be even, and this is equivalent in a finite graph to insisting that each vertex has even degree. 

If we consider Hamiltonian circles, which again meet every finite cut of an infinite graph in an even number of edges, then clearly \ref{itm:parity} is again necessary for a decomposition into Hamiltonian circles to exist, and a natural generalisation of Alspach's conjecture would be that \ref{itm:parity} is also sufficient. Theorem \ref{thm:main-circle} shows that this is true for $4$-regular Cayley graphs.

A Hamiltonian double-ray, however, meets a finite cut of an infinite graph an even number of times if the cut has one infinite component graph, and an odd number of times if it has one infinite component. Hence in order for a decomposition into $k$ Hamiltonian double-rays to exist the number of edges in every finite cut with one infinite component must be even, and the number of edges in every finite cut with two infinite components must have the same parity as $k$. A simple double counting argument shows that the first condition will always hold if $G$ is $2k$-regular, and so perhaps a natural generalisation of Alspach's conjecture for double-rays would be that, together with $2k$-regularity, this second condition is also sufficient. Again Theorem \ref{thm:main} shows that this is true for $4$-regular Cayley graphs.

In light of Theorem \ref{thm:split},
an even more ambitious conjecture,  would be the following. Let 
\begin{enumerate}[label = (Q$i$)]
\item \label{itm:par} Every finite cut $F$ with two infinite components satisfies $|F| \equiv i \mod 2$.
\end{enumerate}

\begin{con}
Let $G$ be a $2k$-regular Cayley graph of an abelian group. If $G$ satisfies \ref{itm:par} then $G$ has a decomposition into $i$ many Hamiltonian double-rays and $k-i$ many Hamiltonian circles.
\end{con}

\bibliographystyle{abbrv}
\bibliography{cycles}

\end{document}